\documentclass[12pt,twoside]{amsart}
\usepackage{latexsym,amsmath,amsopn,amssymb,amsthm,amsfonts}
\usepackage[T2A]{fontenc}
\usepackage[cp1251]{inputenc}
\usepackage[russian,english]{babel}
\usepackage{graphicx}

\newtheorem{theorem}{Theorem}

\newtheorem{proposition}{Proposition}

\newtheorem{corollary}{Corollary}

\newtheorem{lemma}{Lemma}

\newcommand{\Ar}{\operatorname{arctg}}
\newcommand{\Arg}{\operatorname{arg}}

\newcommand{\Lin}{\operatorname{Lin}}

\newcommand{\Arcsin}{\operatorname{arcsin}}
\newcommand{\R}{\operatorname{Re}}
\newcommand{\I}{\operatorname{Im}}
\newcommand{\sgn}{\operatorname{sgn}}

\begin{document}
\begin{flushleft}
УДК 519.46 + 514.763 + 512.81 + 519.9 + 517.911
\end{flushleft}
\begin{flushleft}
MSC 22E30, 49J15, 53C17
\end{flushleft}

\title[Sub-Riemannian distance on $SU(2)$ and $SO(3)$]{Sub-Riemannian distance on Lie groups $SU(2)$ and $SO(3)$}
\author{V.\,N.\,Berestovskii, I.\,A.\,Zubareva,}
\thanks{The work is partially supported by the Russian Foundation for Basic Research (Grant 14-01-00068-a), a grant of the Government of the Russian federation for the State Support of Scientific Research (Agreement 14.B25.31.0029), and State
Maintenance Program for the Leading Scientific Schools of the Russian Federation (Grant NSh-2263.2014.1) }
\address{V.N.Berestovskii}
\address{Sobolev Institute of Mathematics SD RAS, Novosibirsk}
\email{berestov@inbox.ru}
\address{I.A.Zubareva}
\address{Sobolev Institute of Mathematics SD RAS, Omsk Branch}
\email{i\_gribanova@mail.ru}
\maketitle
\maketitle {\small
\begin{quote}
\noindent{\sc Abstract.}
The authors compute distances between arbitrary elements of Lie groups
$SU(2)$ and $SO(3)$ for special left-invariant sub-Riemannian metrics $\rho$ and $d.$ To compute distances for the second metric, we essentially use the fact that canonical two-sheeted covering epimorphism $\Omega$ of the Lie group $SU(2)$ onto the Lie group $SO(3)$ is submetry and local isometry with respect to metrics 
$\rho$ and $d.$ Proofs are based on previously known formulas for geodesics with
origin at the unit, F.~Klein's formula for $\Omega$, trigonometric functions and relations between them, usual Calculus for functions of one real variable. But it is required sufficiently careful application of this simple tool.
\end{quote}}

{\small
\begin{quote}
\end{quote}}

\section*{Introduction}

In this paper, we find distances between arbitrary elements of Lie groups
$SU(2)$ and $SO(3)$ supplied with left-invariant sub-Riemannian metrics $\rho$ and 
$d$, which are also invariant under right shifts by elements of a subgroup $SO(2)$
of the groups $SU(2)$ and $SO(3).$ Additionally, the canonical two-sheeted covering epimorphism $\Omega$ of the Lie group $SU(2)$ onto the Lie group $SO(3)$ is a submetry \cite{BG} and local isometry relative to metrics $\rho$ and $d.$ This fact
plays essential role and permits, though indirectly and after not simple
computation, to find distances in $(SO(3),d)$ using distances in 
$(SU(2),\rho).$

We apply in computations formulas from \cite{BR} (see also \cite{BR1}) for geodesics with origin at the unit on $(SU(2),\rho)$, F.~Klein's formula from \cite{Kl} for $\Omega$, trigonometric functions and relations between them, usual Calculus for functions of one real variable. But it is necessary to apply these simple means rather 
delicately.

Let us note that with the exception of very rare cases, it is unknown how to compute exact distances between points even for homogeneous Riemannian manifolds. By author's opinion, this is impossible in principle in overwhelming majority of cases for homogeneous (sub-)Riemannian manifolds even of small dimension.

The problem under consideration is closely connected with the search problems of
shortest arcs and exact forms of spheres in homogeneous (sub-)Riemannian manifolds. It is even possible that solution of any of these problems permits to solve
somehow more or less easily other two problems.

Exact forms of spheres for left-invariant sub-Riemannian metrics on three-dimensio\-nal Heisenberg group and special left-invariant sub-Riemannian metrics
on Lie groups $SL_2(\mathbb{R})/(\pm E_2)$ and $SO(3)$ are found in paper \cite{BerZ}. Geodesics and shortest arcs of these metrics on the last two
Lie groups are found respectively in papers \cite{Ber1} and \cite{BerZub}.

This paper corrects a mistake in formulas for distances
on $(SU(2),\rho)$ from \cite{BR}; the cut and conjugate locus for 
$(SU(2),\rho)$ and $(SO(3),d)$ are found correctly in \cite{BR}. In more detail, 
we discuss this in section \ref{mis} of this paper.

\section{Preliminaries}

The following statement is proved in theorem 4 from paper \cite{Ber}.

\begin{theorem}
\label{general}
Let $\gamma=\gamma(t),$ $t\in [0,a]\subset \mathbb{R},$ be a parametrized by arclength normal geodesic, i.e. locally shortest arc, on a connected Lie group
$G$ with left-invariant sub-Riemannian metric $d$, defined by completely nonholonomic left-invariant distribution $D$ and a scalar product 
$\langle\cdot,\cdot\rangle$ on $D(e);$ a scalar product $(\cdot,\cdot)$ on
the Lie algebra $\frak{g}$ of Lie group $G$ is chosen so that its restrition to
$D(e)$ coincides with $\langle\cdot,\cdot\rangle$. Then $\gamma=\gamma(t)$ satisfies the following system of ODE
\begin{equation}
\label{dxsp2}
\stackrel{\cdot}\gamma(t)=\gamma(t)u(t),\,\,u(t)\in D(e)\subset \frak{g},\,\,\langle u(t),u(t)\rangle\equiv 1,
\end{equation}
\begin{equation}
\label{vot2}
\psi=\psi(t)=u(t)+v(t), v=v(t)\in \mathfrak{g}\cap D(e)^{\perp},
\end{equation}
\begin{equation}
\label{vot3}
(\stackrel{\cdot}\psi(t),w)=(\psi(t),[u(t),w]), w\in \mathfrak{g},
\end{equation}
where $\psi=\psi(t)\in \mathfrak{g},$ and hence, $u=u(t),$ $v=v(t)\in \frak{g},$ $(v(t),D(e))\equiv 0,$ $t\in [0,a],$ are some real-analytic vector functions.
\end{theorem}

\begin{corollary}
\label{cor}
If in the Lie group $G$ with left-invariant sub-Riemannian metric $d$ two points
are joined by two different parametrized by arclength normal geodesic of
equal length, then any of these geodesics either is not shortest arc or is not  
a part of longer shortest arc.
\end{corollary}

\begin{proof}
Every Lie group has a canonical structure of real-analytic manifold, with respect to which group operations are real-analytic. Therefore in consequence of theorem
\ref{general} any parametrized by arclength normal geodesic in the Lie group with
left-invariant sub-Riemannian metric is real-analytic.

Let suppose now that parametrized by arclength geodesics 
$\gamma_i(t),$ $0\leq t\leq t_0,$ $i=1,2$ в $(G,d)$ join the same points and
$\gamma_1(t),$ $0\leq t\leq t_0+\varepsilon,$  where $\varepsilon > 0,$ is
parametrized by arclength shortest arc. Then concatenation of the geodesic
$\gamma_2(t),$ $0\leq t\leq t_0,$ and $\gamma_1(t),$
$t_0\leq t\leq t_0+ \varepsilon,$ is a shortest arc, hence geodesic. 
Consequently original geodesics, being real-analytic prolongations of joint geodesic segment $\gamma_1(t),$ $t_0\leq t\leq t_0+ \varepsilon,$ must coincide.

The corollary \ref{cor} immediately follows from here.
\end{proof}

\begin{corollary}
\label{cor1}
The statement of corollary \ref{cor} is true for any parametrized by arclength geodesics of three-dimensional Lie group with left-invariant sub-Riemannian metric.
\end{corollary}

\begin{proof}
It is well known that every parametrized by arclength geodesic in three-dimensional
Lie group with left-invariant sub-Riemannian metric is normal (see theorem 3 in \cite{Ber}).  It is remain to apply corollary \ref{cor}.
\end{proof}

\section{Sub-Riemannian distance in the Lie group $SU(2)$}

Every geodesic of left-invariant sub-Riemannian matric on the Lie group is a left shift of some geodesic with origin at the unit. Therefore later we consider only geodesic with origin at the unit.

Let us recall that $SU(2)$ is compact simply connected Lie group of all unitary unimodular $2\times 2$--matrices
$$SU(2)=\left\{\left(\begin{array}{cc}
A & B \\
-\overline{B} & \overline{A}
\end{array}\right)\,\,\mid\,\,A,\,B\in\mathbb{C},\,\,\mid A\mid^2+\mid B\mid^2=1\right\}.$$
Later $(A,B)$ will denote the matrix $\left(\begin{array}{cc}
A & B \\
-\overline{B} & \overline{A}
\end{array}\right)\in SU(2)$.

The Lie algebra $\frak{su}(2)$ of the Lie group $SU(2)$ is the Lie algebra of all skew-Hermitian $2\times 2$--matrices with zero trace
$$\frak{su}(2)=\left\{\left(\begin{array}{cc}
iX & Y \\
-\overline{Y} & -iX
\end{array}\right)\,\,\mid\,\,X\in\mathbb{R},\,\,Y\in\mathbb{C}\right\}.$$

Let us set a basis of $\frak{su}(2)$ in the following way:
\begin{equation}
\label{pp}
p_1=\frac{1}{2}\left(\begin{array}{cc}
0 & 1 \\
-1 & 0
\end{array}\right),\,\,p_2=\frac{1}{2}\left(\begin{array}{cc}
0 & i \\
i & 0
\end{array}\right),\,\,k=\frac{1}{2}\left(\begin{array}{cc}
i & 0 \\
0 & -i
\end{array}\right).
\end{equation}

Let $\Delta(e)=\mbox{Lin}(p_1,p_2),$ where $e$ is the unit of group, $\Lin$
denotes linear span, and on $\Delta(e)$ is given scalar product 
$\left(\cdot,\cdot\right)$ with orthonormal basis $p_1$, $p_2$.
It follows from (\ref{pp}) that
$$[p_1,p_2]=k,\,\,[p_2,k]=p_1,\,\,[k,p_1]=p_2.$$
Therefore left-invariant distribution $\Delta$ on the Lie group $SU(2)$ with given
$\Delta(e)$ is completely nonholonomic and the pair 
$\left(\Delta(e),\left(\cdot,\cdot\right)\right)$ defines left-invariant sub-Riemannian metric $\rho$ on $SU(2),$ which is invariant relative to right
shifts of the group $SU(2)$ by elements of subgroup $SO(2)\subset SU(2)$. In addition every parametrized by arclength geodesic $\gamma=\gamma(t)$, $t\in\mathbb{R}$, in $SU(2)$ with condition $\gamma(0)=e$ is a product of two 1-parameter subgroups (see eg. \cite{AS}):
\begin{equation}
\label{geod2}
\gamma(t)=\exp{(t(\cos{\phi_0}p_1+\sin{\phi_0}p_2+\beta k))}\exp{(-t\beta k)},
\end{equation}
where $\phi_0$, $\beta$ are some real constants.

In consequence of left invariance of the metric $\rho$ on $SU(2)$ it is sufficient to compute distances $\rho(e,g)$ for $e,g\in SU(2).$

\begin{proposition}
\label{bou}
If $g=\gamma(t_0),$ where $\gamma(t)$ is defined by formula (\ref{geod2}), and
$\gamma(t),$ $0\leq t \leq t_0,$ is shortest arc (joining $e$ and $g$) then $t_0\leq \frac{2\pi}{\sqrt{1+\beta^2}}.$
\end{proposition}

\begin{proof} Let us set $A=A(\gamma(t)),$ $B=B(\gamma(t)).$ It is computed in the paper~\cite{BR} that
\begin{equation}
\label{sistema}
\left\{\begin{array}{rl}
\R(A)=\frac{\beta}{\sqrt{1+\beta^2}}\sin{\frac{t\sqrt{1+\beta^2}}{2}}\sin{\frac{\beta t}{2}}+\cos{\frac{t\sqrt{1+\beta^2}}{2}}\cos{\frac{\beta t}{2}}, \\
\I(A)=\frac{\beta}{\sqrt{1+\beta^2}}\sin{\frac{t\sqrt{1+\beta^2}}{2}}\cos{\frac{\beta t}{2}}-\cos{\frac{t\sqrt{1+\beta^2}}{2}}\sin{\frac{\beta t}{2}},
\end{array}\right.
\end{equation}
\begin{equation}
\label{bb}
B=\frac{\sin{\frac{t\sqrt{1+\beta^2}}{2}}}{\sqrt{1+\beta^2}}\left(\cos{\left(\frac{\beta t}{2}+\phi_0\right)}+i\sin{\left(\frac{\beta t}{2}+\phi_0\right)}\right).
\end{equation}

It follows from formulas (\ref{sistema}), (\ref{bb}) that $A$ does not depend on
$\phi_0,$ and $B(\gamma(2\pi/\sqrt{1+\beta^2}))=0.$ Then
$\gamma(2\pi/\sqrt{1+\beta^2})$ does not depend on $\phi_0.$ It is remain to apply corollary \ref{cor1}.
\end{proof}

The main result of this section constitutes

\begin{theorem}
\label{mainn}
Let $g=(A,B)\in SU(2)$, $e=(1,0)$ be the unit of the group $SU(2)$, $t=\rho(g,e)$. Then

1. If $A=0$ then $t=\pi$.

2. If $\mid A\mid=1$ then $t=2\sqrt{\mid\Arg(A)\mid\left(2\pi-\mid\Arg(A)\mid\right)},$ where $\Arg(A)\in [-\pi,\pi]$.

3. If $0<\mid A\mid<1$ and $\R(A)=\mid A\mid\sin{\left(\frac{\pi}{2}\mid A\mid\right)}$, then $t=\pi\sqrt{1-\mid A\mid^2}$.

4. If $0<\mid A\mid<1$ and $\R(A)>\mid A\mid\sin{\left(\frac{\pi}{2}\mid A\mid\right)}$, then
\begin{equation}
\label{tet1}
t=\frac{2}{\sqrt{1+\beta^2}}\Arcsin\sqrt{(1-\mid A\mid^2)(1+\beta^2)}\in\left(0,\frac{\pi}{\sqrt{1+\beta^2}}\right),
\end{equation}
where $\beta$ is unique solution for system of equations
\begin{equation}
\label{a1}
\left\{\begin{array}{c}
\cos{\left(-\frac{\beta}{\sqrt{1+\beta^2}}\Arcsin\sqrt{(1-\mid A\mid^2)(1+\beta^2)}+\Arcsin\frac{\beta\sqrt{1-\mid A\mid^2}}{\mid A\mid}\right)}=\frac{\R(A)}{\mid A\mid}, \\
\sin{\left(-\frac{\beta}{\sqrt{1+\beta^2}}\Arcsin\sqrt{(1-\mid A\mid^2)(1+\beta^2)}+\Arcsin\frac{\beta\sqrt{1-\mid A\mid^2}}{\mid A\mid}\right)}=\frac{\I(A)}{\mid A\mid}.
\end{array}\right.
\end{equation}

5. If $0<\mid A\mid<1$ и $\R(A)<\mid A\mid\sin{\left(\frac{\pi}{2}\mid A\mid\right)}$, then
\begin{equation}
\label{tet2}
t=\frac{2}{\sqrt{1+\beta^2}}\left(\pi-\Arcsin\sqrt{(1-\mid A\mid^2)(1+\beta^2)}\right)\in\left(\frac{\pi}{\sqrt{1+\beta^2}},\frac{2\pi}{\sqrt{1+\beta^2}}\right),
\end{equation}
where $\beta$ is unique solution for system of equations
\begin{equation}
\label{a2}
\left\{\begin{array}{c}
\cos{\left(\frac{\beta}{\sqrt{1+\beta^2}}\left(\pi-\Arcsin\sqrt{(1-\mid A\mid^2)(1+\beta^2)}\right)+\Arcsin\frac{\beta\sqrt{1-\mid A\mid^2}}{\mid A\mid}\right)}=-\frac{\R(A)}{\mid A\mid}, \\
\sin{\left(\frac{\beta}{\sqrt{1+\beta^2}}\left(\pi-\Arcsin\sqrt{(1-\mid A\mid^2)(1+\beta^2)}\right)+\Arcsin\frac{\beta\sqrt{1-\mid A\mid^2}}{\mid A\mid}\right)}=\frac{\I(A)}{\mid A\mid}.
\end{array}\right.
\end{equation}
\end{theorem}

\begin{proof}

If $g=(A,B)\in SU(2)$ and $t=\rho(g,e)$, then there exists a geodesic $\gamma$ 
(see (\ref{geod2})) such that $\gamma(0)=e$, $\gamma(t)=g$.

In consequence of (\ref{sistema}),
\begin{equation}
\label{sss}
\mid A\mid^2=\R^2(A)+\I^2(A)=\frac{\beta^2+\cos^2{\frac{t\sqrt{1+\beta^2}}{2}}}{1+\beta^2}.
\end{equation}
It follows from here and inequality $0\leq t\leq\frac{2\pi}{\sqrt{1+\beta^2}}$ of proposition \ref{bou} that
\begin{equation}
\label{mv}
\sin{\frac{t\sqrt{1+\beta^2}}{2}}=\sqrt{(1-\mid A\mid^2)(1+\beta^2)}.
\end{equation}

Let $\mid A\mid=0$. Then (\ref{mv}) implies that $\beta=0$, $t=\pi$.

Let $\mid A\mid=1$. Then $\sin{\frac{t\sqrt{1+\beta^2}}{2}}=0$ and $B=0$ on the ground of (\ref{mv}), (\ref{bb}). If $A=1$ then $t=0$. If $A\neq 1$ then $t=\frac{2\pi}{\sqrt{1+\beta^2}}$, whence $\mid\beta\mid=\frac{\sqrt{4\pi^2-t^2}}{t}$. Therefore equations of the system (\ref{sistema}) take the form
\begin{equation}
\label{mmm}
\R(A)=-\cos{\frac{\pi\beta}{\sqrt{1+\beta^2}}},\quad
\I(A)=\sin{\frac{\pi\beta}{\sqrt{1+\beta^2}}}.
\end{equation}
Substituting $\mid\beta\mid=\frac{\sqrt{4\pi^2-t^2}}{t}$ into (\ref{mmm}), we get
$$\R(A)=-\cos{\frac{\sqrt{4\pi^2-t^2}}{2}},\quad \I(A)=\pm\sin{\frac{\sqrt{4\pi^2-t^2}}{2}}.$$
Consequently $\sqrt{4\pi^2-t^2}=2\left(\pi-\mid\mbox{arg}(A)\mid\right)$, where 
$\mbox{arg}(A)\in[-\pi,\pi]$.
Expressing $t$ from the last equality, we get the statement p. 2 of theorem \ref{mainn}.

Let $0<\mid A\mid<1$. Then it follows from (\ref{sistema}) and proposition
\ref{bou} that  $0<t<\frac{2\pi}{\sqrt{1+\beta^2}}$. In consequence of (\ref{sistema}), (\ref{sss}), there exists $\gamma\in[-\pi,\pi]$, satisfying conditions
\begin{equation}
\label{gama}
\cos{\gamma}=\frac{\sqrt{1+\beta^2}\cos{\frac{t\sqrt{1+\beta^2}}{2}}}{\sqrt{\beta^2+\cos^2{\frac{t\sqrt{1+\beta^2}}{2}}}},\quad
\sin{\gamma}=\frac{\beta\sin{\frac{t\sqrt{1+\beta^2}}{2}}}{\sqrt{\beta^2+\cos^2{\frac{t\sqrt{1+\beta^2}}{2}}}}.
\end{equation}
In addition, one can write the system (\ref{sistema}) in the form
\begin{equation}
\label{sistema0}
\cos{\left(\gamma-\frac{\beta t}{2}\right)}=\frac{\R(A)}{\mid A\mid},\quad
\sin{\left(\gamma-\frac{\beta t}{2}\right)}=\frac{\I(A)}{\mid A\mid}.
\end{equation}

Let us consider the case $0<t<\frac{\pi}{\sqrt{1+\beta^2}}$. It follows from
(\ref{mv}) that one can find $t$ by formula (\ref{tet1}).
On the ground of (\ref{gama}), $\cos{\gamma}>0$. Finding $\gamma$ rom
(\ref{gama}) and using (\ref{tet1}), (\ref{sss}), we get
$$\gamma=\mbox{arcsin}\left(\frac{\beta\sin{\frac{t\sqrt{1+\beta^2}}{2}}}{\sqrt{\beta^2+\cos^2{\frac{t\sqrt{1+\beta^2}}{2}}}}\right)=
\mbox{arcsin}\frac{\beta\sqrt{1-\mid A\mid^2}}{\mid A\mid}.$$
Therefore one can write the system (\ref{sistema0}) in the form (\ref{a1}).

\begin{lemma}
\label{t1}
Let us suppose that $0<\mid A\mid<1$ and define function $t_1=t_1(\beta)$ on segment
 $\left[-\frac{\mid A\mid}{\sqrt{1-\mid A\mid^2}},\frac{\mid A\mid}{\sqrt{1-\mid A\mid^2}}\right]$  by formula (\ref{tet1}). Then

1. $t_1(\beta)$ is even function;

2. $t_1(\beta)$ increases on segment 
$\left[0,\frac{\mid A\mid}{\sqrt{1-\mid A\mid^2}}\right]$
and decreases on segment $\left[-\frac{\mid A\mid}{\sqrt{1-\mid A\mid^2}},0\right]$;

3. Its range is segment 
$\left [2\Arcsin\sqrt{1-\mid A\mid^2},\pi\sqrt{1-\mid A\mid^2}\right ]$.
\end{lemma}

\begin{proof}
The evenness of the function $t_1=t_1(\beta)$ is evident.

Let us set
\begin{equation}
\label{z}
z=\sqrt{(1-\mid A\mid^2)(1+\beta^2)},\,\,z\in\left [\sqrt{1-\mid A\mid^2},1\right ].
\end{equation}
It is not difficult to show that
\begin{equation}
\label{tt}
t_1^{\prime}(\beta)=\frac{2\beta}{(\sqrt{1+\beta^2})^3}\cdot\frac{z-\sqrt{1-z^2}\Arcsin z}{\sqrt{1-z^2}}.
\end{equation}
Let us consider function
\begin{equation}
\label{fff}
f(z)=z-\sqrt{1-z^2}\Arcsin z,\,\,z\in\left [\sqrt{1-\mid A\mid^2},1\right ].
\end{equation}
Since
$$f(0)=0,\,\,f^{\prime}(z)=\frac{z\Arcsin z}{\sqrt{1-z^2}}>0\,\,\mbox{if}\,\,z\in (0,1),$$
then $f(z)>0$ for every $z\in\left [\sqrt{1-\mid A\mid^2},1\right ].$
The statement p.\,2 of lemma \ref{t1} follows from here and (\ref{tt}). 
On the ground of (\ref{tet1}),
$$t_1(0)=2\Arcsin\sqrt{1-\mid A\mid^2},\quad t_1\left(\frac{\mid A\mid}{\sqrt{1-\mid A\mid^2}}\right)=\pi\sqrt{1-\mid A\mid^2}.$$
These equalities and p.~2 of lemma \ref{t1} imply the statement p.\,3 of lemma \ref{t1}.
\end{proof}

\begin{lemma}
\label{f1}
Let us suppose $0<\mid A\mid<1$ and define function $F_1(\beta)$ on the segment 
$\left [-\frac{\mid A\mid}{\sqrt{1-\mid A\mid^2}},\frac{\mid A\mid}{\sqrt{1-\mid A\mid^2}}\right ]$ by formula
\begin{equation}
\label{arg1}
F_{1}(\beta)=-\frac{\beta}{\sqrt{1+\beta^2}}\Arcsin\sqrt{(1-\mid A\mid^2)(1+\beta^2)}+\Arcsin\frac{\beta\sqrt{1-\mid A\mid^2}}{\mid A\mid}.
\end{equation}
Then

1. $F_1(\beta)$ is odd function;

2. $F_1(\beta)$ increases on all its domain;

3. Its range is segment
$\left [\frac{\pi}{2}(\mid A\mid-1),\frac{\pi}{2}(1-\mid A\mid)\right ]$.
\end{lemma}

\begin{proof}
The oddness of the function $F_{1}(\beta)$ is evident.

It is not difficult to show that
$$F_{1}^{\prime}(\beta)=\frac{f(z)}{(\sqrt{1+\beta^2})^3\cdot\sqrt{1-z^2}},$$
where $z$ is defined by formula (\ref{z}), $f(z)$ is defined by formula (\ref{fff}). Now p.~2 of lemma \ref{f1} follows from the fact that $f(z)>0$ for
every
$z\in\left [\sqrt{1-\mid A\mid^2},1\right ].$

On the ground of (\ref{arg1}),
$$F_{1}\left(\frac{\mid A\mid}{\sqrt{1-\mid A\mid^2}}\right)=\frac{\pi}{2}(1-\mid A\mid).$$
The statement p.\,3 of lemma \ref{f1} follows from this equality and p.~2 of lemma \ref{f1}. 
\end{proof}

The function $F_1(\beta)$, defined by formula (\ref{arg1}), is the argument
of the cosine and the sinus functions in the left parts of equations (\ref{a1}).
Then on the ground of lemma~\ref{f1} the system (\ref{a1}) has a solution,
and this solution is unique, if and only if 
\begin{equation}
\label{ner1}
\frac{\R(A)}{\mid A\mid}\geq\cos{\left(\frac{\pi}{2}(1-\mid A\mid)\right)}\quad\Leftrightarrow\quad\R(A)\geq \mid A\mid\sin{\left(\frac{\pi}{2}\mid A\mid\right)}.
\end{equation}

Notice that if in the formula (\ref{ner1}) stands the equlity sign then 
$\mid\beta\mid=\frac{\mid A\mid}{\sqrt{1-\mid A\mid^2}}$.
But then $(1-\mid A\mid^2)(1+\beta^2)=1,$ and on the ground of 
(\ref{mv}), $\sin{\frac{t\sqrt{1+\beta^2}}{2}}=1$, which is impossible if 
$0<t<\frac{\pi}{\sqrt{1+\beta^2}}$.

Let suppose now that 
$\frac{\pi}{\sqrt{1+\beta^2}}<t<\frac{2\pi}{\sqrt{1+\beta^2}}$. It follows from
(\ref{mv}) that one can find $t$ by the formula (\ref{tet2}).
Then $\cos{\gamma}<0$ on the ground of (\ref{gama}). Finding $\gamma$ from (\ref{gama}) and using (\ref{tet2}), (\ref{sss}), we get
$$\gamma=\pi-\Arcsin\left(\frac{\beta\sin{\frac{t\sqrt{1+\beta^2}}{2}}}{\sqrt{\beta^2+\cos^2{\frac{t\sqrt{1+\beta^2}}{2}}}}\right)=
\pi-\Arcsin\frac{\beta\sqrt{1-\mid A\mid^2}}{\mid A\mid}.$$
Therefore one can write the system of equations (\ref{sistema0}) in the form (\ref{a2}).

\begin{lemma}
\label{t2}
Let us supppose that $0<\mid A\mid<1$ and define function $t_2=t_2(\beta)$ on
the segment 
$\left[-\frac{\mid A\mid}{\sqrt{1-\mid A\mid^2}},\frac{\mid A\mid}{\sqrt{1-\mid A\mid^2}}\right]$  by formula (\ref{tet2}). Then

1. $t_2(\beta)$ is even function;

2. $t_2(\beta)$ decreases on the segment 
$\left[0,\frac{\mid A\mid}{\sqrt{1-\mid A\mid^2}}\right]$
and increases on the segment 
$\left[-\frac{\mid A\mid}{\sqrt{1-\mid A\mid^2}},0\right]$;

3. Its range is segment 
$\left [\pi\sqrt{1-\mid A\mid^2},2(\pi-\Arcsin\sqrt{1-\mid A\mid^2})\right ]$.
\end{lemma}

\begin{proof}
In consequence of (\ref{tet1}) and (\ref{tet2}),
$$t_2(\beta)=\frac{2\pi}{\sqrt{1+\beta^2}}-t_1(\beta).$$
Now lemma \ref{t2} follows from lemma \ref{t1} and the fact that
$$t_2(0)=2(\pi-\Arcsin\sqrt{1-\mid A\mid^2}),\quad t_2\left(\frac{\mid A\mid}{\sqrt{1-\mid A\mid^2}}\right)=\pi\sqrt{1-\mid A\mid^2}.$$
\end{proof}

\begin{lemma}
\label{f2}
Let us suppose that $0<\mid A\mid<1$ and define function $F_2(\beta)$ on the segment 
$\left [-\frac{\mid A\mid}{\sqrt{1-\mid A\mid^2}},\frac{\mid A\mid}{\sqrt{1-\mid A\mid^2}}\right ]$ by formula
$$F_{2}(\beta)=\frac{\beta}{\sqrt{1+\beta^2}}\left(\pi-\Arcsin\sqrt{(1-\mid A\mid^2)(1+\beta^2)}\right)+\Arcsin\frac{\beta\sqrt{1-\mid A\mid^2}}{\mid A\mid}.$$
Then

1. $F_2(\beta)$ is odd funtion;

2. $F_2(\beta)$ increases on all its domain;

3. Its range is segement 
$\left [-\frac{\pi}{2}(1+\mid A\mid),\frac{\pi}{2}(1+\mid A\mid)\right ]$.
\end{lemma}

\begin{proof}
The oddness of function $F_{2}(\beta)$ is evident.

It follows from (\ref{arg1}) that
$$F_{2}(\beta)=\frac{\pi\beta}{\sqrt{1+\beta^2}}+F_{1}(\beta).$$
Now lemma \ref{f2} follows from lemma \ref{f1} and the fact that
$$F_{2}\left(-\frac{\mid A\mid}{\sqrt{1-\mid A\mid^2}}\right)=-\frac{\pi}{2}(1+\mid A\mid),\,\,
F_{2}\left(\frac{\mid A\mid}{\sqrt{1-\mid A\mid^2}}\right)=\frac{\pi}{2}(1+\mid A\mid).$$
\end{proof}

The function $F_2(\beta)$, defined in lemma~\ref{f2}, is argument of the cosine and the sinus functions in left parts of equations (\ref{a2}).
Therefore, on the ground of lemma~\ref{f2}, the system (\ref{a2}) has a solution,
and this solution is unique, if and only if
\begin{equation}
\label{ner2}
-\frac{\R(A)}{\mid A\mid}\geq\cos{\left(\frac{\pi}{2}(1+\mid A\mid)\right)}\quad\Leftrightarrow\quad\R(A)\leq \mid A\mid\sin{\left(\frac{\pi}{2}\mid A\mid\right)}.
\end{equation}

Let us note that if in the formula (\ref{ner2}) stands the equality sign, then
$\mid\beta\mid=\frac{\mid A\mid}{\sqrt{1-\mid A\mid^2}}$.
But then $(1-\mid A\mid^2)(1+\beta^2)=1$, and 
$\sin{\frac{t\sqrt{1+\beta^2}}{2}}=1$ on the ground of (\ref{mv}), 
which is impossible if 
$\frac{\pi}{\sqrt{1+\beta^2}}<t<\frac{2\pi}{\sqrt{1+\beta^2}}$.

It is remain to consider the case $t=\frac{\pi}{\sqrt{1+\beta^2}}$.  Then 
$\mid \beta\mid=\frac{\mid A\mid}{\sqrt{1-\mid A\mid^2}}$ in consequence of (\ref{mv}). Substituting the last formula into expression for $t$, we get that 
$t=\pi\sqrt{1-\mid A\mid^2}$. In view of (\ref{gama}), $\gamma=\frac{\pi}{2}$ if 
$\beta>0,$ and $\gamma=-\frac{\pi}{2}$ if $\beta<0$. Consequently one can write the equations (\ref{sistema0}) in the form
$$\sin{\frac{\beta t}{2}}=\sgn(\beta)\frac{\R(A)}{\mid A\mid},\quad \cos{\frac{\beta t}{2}}=\sgn(\beta)\frac{\I(A)}{\mid A\mid}.$$
Substituting expressions for $\mid\beta\mid$ and $t$ into the last equations, we get that they are equivalent to the equality 
$\R(A)=\mid A\mid\sin{\left(\frac{\pi}{2}\mid A\mid\right)}$ from p. 3 of
theorem \ref{mainn}. This finishes the proof of theorem \ref{mainn}.
\end{proof}

\section{Sub-Riemannian distance in the Lie group $SO(3)$}

Let us recall that $SO(3)$ is compact Lie group, consisting of all orthogonal   $3\times 3$--matrices with determinant $1$.
Its Lie algebra $\frak{so}(3)$ consists of all real skew-symmetric  $3\times 3$--matrices. Let us define a basis of $\frak{so}(3)$ as follows:
\begin{equation}
\label{abc}
a=\left(\begin{array}{ccc}
0 & -1 & 0 \\
1 & 0 & 0 \\
0 & 0 & 0
\end{array}\right),\quad
b=\left(\begin{array}{ccc}
0 & 0 & -1 \\
0 & 0 & 0 \\
1 & 0 & 0
\end{array}\right),\quad
c=\left(\begin{array}{ccc}
0 & 0 & 0 \\
0 & 0 & -1 \\
0 & 1 & 0
\end{array}\right).
\end{equation}
Let us set $D(e)=\Lin(a,b)$ and define scalar product $\langle\cdot,\cdot\rangle$
on $D(e)$ with orthonormal basis $a,b.$ It follows from (\ref{abc}) that
$$[a,b]=c,\quad [b,c]=a,\quad [c,a]= b.$$
Therefore left-invariant distribution $D$ on the Lie group $SO(3)$ with given
$D(e)$ is completely nonholonomic and pair $(D(e),\langle\cdot,\cdot\rangle)$ defines a left-invariant sub-Riemannian metric $d$ on $SO(3).$

The main result of this section is given by the following theorem.

\begin{theorem}
\label{mainnn}
Let $C=(c_{ij})\in SO(3)$, $E$ be the unit of the group $SO(3)$, $t=d(C,E)$. Then

1. If $c_{11}=-1$ then $t=\pi$.

2. If $c_{11}=1$ and $C\neq E$ then $t=\frac{2\pi}{\sqrt{1+\beta^2}}$, where 
$\beta$ is unique solution of the system of equations
\begin{equation}
\label{k0}
\left\{\begin{array}{l}
\cos\frac{\pi\beta}{\sqrt{1+\beta^2}}=-\frac{\sqrt{1+c_{11}+c_{22}+c_{33}}}{2}, \\
\sin\frac{\pi\beta}{\sqrt{1+\beta^2}}=\sgn(c_{32}-c_{23})\frac{\sqrt{1+c_{11}-c_{22}-c_{33}}}{2}.
\end{array}\right.
\end{equation}

3. If
$-1<c_{11}<1$ and $\cos{\left(\pi\sqrt{\frac{1+c_{11}}{2}}\right)}=-\frac{c_{22}+c_{33}}{1+c_{11}}$
then $t=\pi\sqrt{\frac{1}{2}(1-c_{11})}.$

4. If $-1<c_{11}<1$ and $\cos{\left(\pi\sqrt{\frac{1+c_{11}}{2}}\right)}>-\frac{c_{22}+c_{33}}{1+c_{11}}$ then
$$t=\frac{2}{\sqrt{1+\beta^2}}\Arcsin\sqrt{\frac{1}{2}(1-c_{11})(1+\beta^2)},$$
where $\beta$ is unique solution of the system of equations 
\begin{equation}
\label{k1}
\tiny{\left\{\begin{array}{l}
\cos{\left(-\frac{\beta}{\sqrt{1+\beta^2}}\Arcsin\sqrt{\frac{1}{2}(1-c_{11})(1+\beta^2)}+\Arcsin\left(\beta\sqrt{\frac{1-c_{11}}{1+c_{11}}}\right)\right)}=
\sqrt{\frac{1+c_{11}+c_{22}+c_{33}}{2(1+c_{11})}}, \\
\sin{\left(-\frac{\beta}{\sqrt{1+\beta^2}}\Arcsin\sqrt{\frac{1}{2}(1-c_{11})(1+\beta^2)}+\Arcsin\left(\beta\sqrt{\frac{1-c_{11}}{1+c_{11}}}\right)\right)}=\sgn(c_{32}-c_{23})\sqrt{\frac{1+c_{11}-c_{22}-c_{33}}{2(1+c_{11})}}.
\end{array}\right.}
\end{equation}

5. If $-1<c_{11}<1$ and $\cos{\left(\pi\sqrt{\frac{1+c_{11}}{2}}\right)}<-\frac{c_{22}+c_{33}}{1+c_{11}}$ then
$$t=\frac{2}{\sqrt{1+\beta^2}}\left(\pi-\Arcsin\sqrt{\frac{1}{2}(1-c_{11})(1+\beta^2)}\right),$$
where $\beta$ is unique solution of the system of equations 
\begin{equation}
\label{aa2}
\tiny{\left\{\begin{array}{l}
\cos{\left(\frac{\beta}{\sqrt{1+\beta^2}}\left(\pi-\Arcsin\sqrt{\frac{1}{2}(1-c_{11})(1+\beta^2)}\right)+\Arcsin\left(\beta\sqrt{\frac{1-c_{11}}{1+c_{11}}}\right)\right)}=-\sqrt{\frac{1+c_{11}+c_{22}+c_{33}}{2(1+c_{11})}}, \\
\sin{\left(\frac{\beta}{\sqrt{1+\beta^2}}\left(\pi-\Arcsin\sqrt{\frac{1}{2}(1-c_{11})(1+\beta^2)}\right)+\Arcsin\left(\beta\sqrt{\frac{1-c_{11}}{1+c_{11}}}\right)\right)}=\sgn(c_{32}-c_{23})\sqrt{\frac{1+c_{11}-c_{22}-c_{33}}{2(1+c_{11})}}.
\end{array}\right.}
\end{equation}
\end{theorem}

\begin{proof}

According to \cite{Post}, the mapping $\Omega_1$, which associates to an elemet $(A,B)$ of the group $SU(2)$ the quaternion $q=A+B{\bf j}$, is an isomorphism of the Lie group $SU(2)$ onto multiplicative group $Sp(1)$ of unit quaternion. Under identification of arbitrary point $(x,y,z)$ in $\mathbb{R}^3$ (with standard scalar product) with quaternion $x{\bf i}+y{\bf j}+z{\bf k},$ the mapping $\Omega_2$, which associates to a quaternion $q\in Sp(1)$  the transformation
$$x{\bf i}+y{\bf j}+z{\bf k}\rightarrow q\cdot(x{\bf i}+y{\bf j}+z{\bf k})\cdot\bar{q},$$
is an epimorphism of the Lie group $Sp(1)$ onto $SO(3)$ (see~\cite{Kl}, c.\,106). Therefore the mapping $\Omega=\Omega_2\circ\Omega_1$ is an epimorphism of $SU(2)$ onto $SO(3)$. The mapping $\Omega:(SU(2),\rho)\rightarrow (SO(3),d)$ is
a submetry \cite{BG} and local isometry. This follows from the fact that its differential $d\Omega(e)$ is an isomorphism of the Lie algebra $\frak{su}(2)$ onto the Lie algebra $\frak{so}(3)$, translating the orthonormal basis $(p_1,p_2)$ of 
$\Delta(e)$ to the orthonormal basis $(-b,a)$ of $D(e);$ in addition, the group 
$I(SO(2))$ of conjugations of element in $SO(3)$ by elements of subgroup 
$SO(2)\subset SO(3)$ is simultaneously a (sub)group of automorphisms of the Lie group $SO(3)$ and isometries of metric space $(SO(3),d),$ while its differential
$Ad(SO(2))=d(I(SO(2)))(e)$ is the group of isometric rotations of Euclidean plane 
$(D(e),\langle\cdot,\cdot\rangle)$ \cite{BerZub}. According to \cite{Kl},
 \begin{equation}
\label{matr}
\Omega(A,B)=\left(\begin{array}{ccc}
A^2_1+A^2_2-B^2_1-B^2_2 & 2(A_2B_1-B_2A_1) & 2(A_2B_2+B_1A_1) \\
2(A_2B_1+B_2A_1) & A^2_1-A^2_2+B^2_1-B^2_2 & 2(B_1B_2-A_1A_2) \\
2(A_2B_2-B_1A_1) & 2(B_2B_1+A_2A_1) & A^2_1-A^2_2-B^2_1+B^2_2
\end{array}\right),
\end{equation}
where
$$A_1=\R(A),\,\,A_2=\I(A),\,\,B_1=\R(B),\,\,B_2=\I(B).$$

Let $C=(c_{ij})\in SO(3)$.  In consequence of (\ref{matr}),
$$\R^2(A)-\I^2(A)=\frac{c_{22}+c_{33}}{2},\quad\mid A\mid^2=\frac{1+c_{11}}{2},\quad 4\R(A)\I(A)=c_{32}-c_{23},$$
whence
$$\mid\R(A)\mid=\frac{\sqrt{1+c_{11}+c_{22}+c_{33}}}{2},\quad
\mid\I(A)\mid=\frac{\sqrt{1+c_{11}-c_{22}-c_{33}}}{2},$$
$$\sgn(\R(A)\I(A))=\sgn(c_{32}-c_{23}).$$

Thus, any element $C=(c_{ij})$ of the group SO(3) is the image of two elements 
$(A,B)$, $(-A,-B)$ of the group $SU(2)$. Without loss of generality, we can assume that the element $(A,B)$ satisfies equations
\begin{equation}
\label{form1}
\R(A)=\frac{\sqrt{1+c_{11}+c_{22}+c_{33}}}{2},\quad\I(A)=\sgn(c_{32}-c_{23})\frac{\sqrt{1+c_{11}-c_{22}-c_{33}}}{2},
\end{equation}
\begin{equation}
\label{form2}
\mid A\mid=\sqrt{\frac{1+c_{11}}{2}}.
\end{equation}
Then $t=d(C,E)$ is equal to the lesser of distances 
$\rho((A,B),e)$ и $\rho((-A,-B),e)$.

Let us suppose that $c_{11}=-1$. It follows from (\ref{form2}) that $A=0$. On the ground of p.~1 in theorem~\ref{mainn} we get that $t=\pi$.

Let us prove p. 2 of theorem \ref{mainnn}. Let $c_{11}=1$. Then it follows from
(\ref{form2}) that $\mid A\mid=1$, $B=0$. On the ground of (\ref{mmm}), the proof of p.\,2 in theorem \ref{mainn}, and (\ref{form1}), the distance $t=d(C,E)$ is
equal to the lesser of numbers $t_i=2\pi/\sqrt{1+\beta_i^2}$, $i=1,2$,
where $\beta_1$ is unique solution of the system of equations (\ref{k0}),
$\beta_2$ is unique solution of the system of equations 
$$\left\{\begin{array}{l}
\cos\frac{\pi\beta}{\sqrt{1+\beta^2}}=\frac{\sqrt{1+c_{11}+c_{22}+c_{33}}}{2}, \\
\sin\frac{\pi\beta}{\sqrt{1+\beta^2}}=\sgn(c_{23}-c_{32})\frac{\sqrt{1+c_{11}-c_{22}-c_{33}}}{2}.
\end{array}\right.$$

It is clear that $F(\beta)=\frac{\pi\beta}{\sqrt{1+\beta^2}}$ is odd increasing function and $F(0)=0$. It follows from the last system of equations, inequality 
$\mid F(\beta)\mid < \pi$, and (\ref{k0}) that
$$\mid F(\beta_2)-F(\beta_1)\mid=\pi,\,\,\cos{F(\beta_1)}<0.$$
Therefore
$$F(\mid\beta_1\mid)=\mid F(\beta_1)\mid>\frac{\pi}{2}>\mid F(\beta_2)\mid=F(\mid\beta_2\mid).$$
Now it follows from increasing of the function $F(\beta)$ that 
$\mid\beta_1\mid>\mid\beta_2\mid$. Thus $t_1<t_2$ and $d(C,E)=t_1$.

Let suppose now that $-1<c_{11}<1$. Then (\ref{form2}) implies that 
$0<\mid A\mid<1$.

Now it is easy to see that
$$\R(A)=\mid A\mid\sin{\left(\frac{\pi}{2}\mid A\mid\right)}\,\,\Leftrightarrow\,\, \cos{\left(\pi\sqrt{\frac{1+c_{11}}{2}}\right)}=-\frac{c_{22}+c_{33}}{1+c_{11}},$$
because in consequence of the first equality and (\ref{form1}), (\ref{form2}),
$$\sin\left(\frac{\pi}{2}\mid A\mid\right)=\cos\arg(A)=\sqrt{\frac{1}{2}\left(1+\frac{c_{22}+c_{33}}{1+c_{11}}\right)},$$
$$\left|\cos\left(\frac{\pi}{2}\mid A\mid\right)\right|= \mid\sin\arg(A)\mid=\sqrt{\frac{1}{2}\left(1-\frac{c_{22}+c_{33}}{1+c_{11}}\right)}.$$

In this case by p.~3 of theorem~\ref{mainn} and (\ref{form2}),
$$d(C,E)=\pi\sqrt{1-\mid A\mid^2}=\pi\sqrt{\frac{1-c_{11}}{2}},$$
and p.~3 of theorem \ref{mainnn} is proved.

It follows from (\ref{form1}), (\ref{form2}) that
$$\R(A)>\mid A\mid\sin{\left(\frac{\pi}{2}\mid A\mid\right)}\,\,\Leftrightarrow\,\,
\cos{\left(\pi\sqrt{\frac{1+c_{11}}{2}}\right)}>-\frac{c_{22}+c_{33}}{1+c_{11}}.$$
In this case for computation of $t_1=\rho((A,B),e)$ one needs to apply p.\,4 of theorem \ref{mainn}, and for computation of $t_2=\rho((-A,-B),e)$ one needs to apply p.\,5 of theorem \ref{mainn}. Therefore $d(C,E)$ is equal to the lesser
of numbers $t_1$ and $t_2$, in addition
$$t_1=\frac{2}{\sqrt{1+\beta_1^2}}\Arcsin\sqrt{\frac{1}{2}(1-c_{11})(1+\beta_1^2)}\in\left(0,\frac{\pi}{\sqrt{1+\beta_1^2}}\right),$$
$$t_2=\frac{2}{\sqrt{1+\beta_2^2}}\left(\pi-\Arcsin\sqrt{\frac{1}{2}(1-c_{11})(1+\beta_2^2)}\right)\in\left(\frac{\pi}{\sqrt{1+\beta_2^2}},\frac{2\pi}{\sqrt{1+\beta_2^2}}\right),$$
where $\beta_1$ is unique solution of the system of equations (\ref{k1}),
$\beta_2$ is unique solution of the system of equations
\begin{equation}
\label{k2}
\tiny{\left\{\begin{array}{l}
\cos{\left(\frac{\beta}{\sqrt{1+\beta^2}}\left(\pi-\Arcsin\sqrt{\frac{1}{2}(1-c_{11})(1+\beta^2)}\right)+\Arcsin\left(\beta\sqrt{\frac{1-c_{11}}{1+c_{11}}}\right)\right)}=
\sqrt{\frac{1+c_{11}+c_{22}+c_{33}}{2(1+c_{11})}}, \\
\sin{\left(\frac{\beta}{\sqrt{1+\beta^2}}\left(\pi-\Arcsin\sqrt{\frac{1}{2}(1-c_{11})(1+\beta^2)}\right)+\Arcsin\left(\beta\sqrt{\frac{1-c_{11}}{1+c_{11}}}\right)\right)}=
\sgn(c_{23}-c_{32})\sqrt{\frac{1+c_{11}-c_{22}-c_{33}}{2(1+c_{11})}}.
\end{array}\right.}
\end{equation}
In consequence of (\ref{k1}), (\ref{k2}) and lemmas~\ref{f1}, \ref{f2}, we have $F_2(\beta_2)=-F_1(\beta_1)$, i.e.
$$F_1(\beta_1)+F_1(\beta_2)=-\frac{\pi\beta_2}{\sqrt{1+\beta_2^2}}.$$

If $\beta_2>0$ then $F_1(\beta_1)<-F_1(\beta_2)=F_1(-\beta_2)<0$. Now it follows from increasing of the function $F_1(\beta)$ (see lemma~\ref{f1}) that 
$\beta_1<-\beta_2<0$. Therefore $\beta_1^2>\beta_2^2$ and $t_1<t_2$.

If $\beta_2<0$ then $F_1(\beta_1)>-F_1(\beta_2)=F_1(-\beta_2)>0$. It follows from here and increasing of the function $F_1(\beta)$ (see lemma~\ref{f1}) that 
$\beta_1>-\beta_2>0$. Therefore $\beta_1^2>\beta_2^2$, whence $\frac{\pi}{\sqrt{1+\beta_1^2}}<\frac{\pi}{\sqrt{1+\beta_2^2}}$ and $t_1<t_2$. Consequently $d(C,E)=t_1$.

This finishes the proof of p.~4 in theorem~\ref{mainn}.

It follows from (\ref{form1}), (\ref{form2}) that
$$\R(A)<\mid A\mid\sin{\left(\frac{\pi}{2}\mid A\mid\right)}\,\,\Leftrightarrow\,\,
\cos{\left(\pi\sqrt{\frac{1+c_{11}}{2}}\right)}<-\frac{c_{22}+c_{33}}{1+c_{11}}.$$
Since $\R(A)\geq 0$, then $-\R(A)<\mid A\mid\sin{\left(\frac{\pi}{2}\mid A\mid\right)}$. In this case for computation of both $t_1=\rho((A,B),e)$ and 
$t_2=\rho((-A,-B),e)$ one needs to apply p.\,5 of theorem \ref{mainn}, i.e.
$$t_i=\frac{2}{\sqrt{1+\beta_i^2}}\left(\pi-\Arcsin\sqrt{\frac{1}{2}(1-c_{11})(1+\beta_i^2)}\right)\in\left(\frac{\pi}{\sqrt{1+\beta_i^2}},\frac{2\pi}{\sqrt{1+\beta_i^2}}\right),\,\,i=1,2,$$
where $\beta_1$ is unique solution of the system (\ref{aa2}),
$\beta_2$ is unique solution of the system of equations
\begin{equation}
\label{k3}
\tiny{\left\{\begin{array}{c}
\cos{\left(\frac{\beta}{\sqrt{1+\beta^2}}\left(\pi-\Arcsin\sqrt{\frac{1}{2}(1-c_{11})(1+\beta^2)}\right)+\Arcsin\left(\beta\sqrt{\frac{1-c_{11}}{1+c_{11}}}\right)\right)}=
\sqrt{\frac{1+c_{11}+c_{22}+c_{33}}{2(1+c_{11})}}, \\
\sin{\left(\frac{\beta}{\sqrt{1+\beta^2}}\left(\pi-\Arcsin\sqrt{\frac{1}{2}(1-c_{11})(1+\beta^2)}\right)+\Arcsin\left(\beta\sqrt{\frac{1-c_{11}}{1+c_{11}}}\right)\right)}=
\sgn(c_{23}-c_{32})\sqrt{\frac{1+c_{11}-c_{22}-c_{33}}{2(1+c_{11})}}.
\end{array}\right.}
\end{equation}
On the ground of (\ref{aa2}), (\ref{k3}), and lemma\ref{f2} (in particular, p. 3 of this lemma implies inequality $\mid F_2(\beta) \mid< \pi$) we get that
$\mid F_2(\beta_2)-F_2(\beta_1)\mid=\pi$. It follows from this equality, 
conditions $F_2(0)=0$, $\cos{F_2(\beta_1)}<0$, and lemma~\ref{f2} that
$$ F_2(\mid\beta_1\mid)=\mid F_2(\beta_1)\mid>\frac{\pi}{2}>\mid F_2(\beta_2)\mid=F_2(\mid\beta_2\mid).$$
Now it follows from increasing of the function $F_2(\beta)$ that 
$\mid\beta_1\mid>\mid\beta_2\mid$. It is remain to note that function
$t(\beta)=\frac{2}{\sqrt{1+\beta^2}}\left(\pi-\mbox{arcsin}\sqrt{\frac{1}{2}(1+\beta^2)(1-c_{11})}\right)$ decreases if $\beta>0$ and increases if $\beta<0$. Therefore $t_1=t(\mid\beta_1\mid)<t(\mid\beta_2\mid)=t_2$ and $d(C,E)=t_1$.

This finishes the proof of p.~5 and all theorem~\ref{mainnn}.

\end{proof}

\section{On the paper \cite{BR}}
\label{mis}

The following theorem is stated in the paper \cite{BR} by U.~Boscain and F.~Rossi 
(see also their paper \cite{BR1}).
\begin{theorem}
\label{bosc}
Let $g=(A,B)\in SU(2)$, $e=(1,0)$ is the unit of the group $SU(2)$. Then
$$\rho(g,e)=\left\{
\begin{array}{rl}
2\sqrt{\Arg(A)(2\pi-\Arg(A))}, & \mbox{если } B=0, \\
\psi(A), & \mbox{если } B\neq 0,
\end{array}\right.$$
where $\Arg(A)\in [0,2\pi]$, and $\psi(A)=t$ is unique solution of the system
\begin{equation}
\label{br}
\left\{\begin{array}{l}-\frac{\beta t}{2}+\Ar\left(\frac{\beta}{\sqrt{1+\beta^2}}\tg\frac{t\sqrt{1+\beta^2}}{2}\right)=\Arg(A), \\
\frac{\sin{\frac{t\sqrt{1+\beta^2}}{2}}}{\sqrt{1+\beta^2}}=\sqrt{1-\mid A\mid^2}, \\
t\in\left(0,\frac{2\pi}{\sqrt{1+\beta^2}}\right).
\end{array}\right.
\end{equation}
\end{theorem}

Let us set $\mbox{arg}(A)=0$, $0<\mid B\mid<1$. Then it is clear that $\beta=0$ is a solution of the first equation of the system (\ref{br}) and the second equation
of the system (\ref{br}) gives two different solutions 
$t=2\mbox{arcsin}\sqrt{1-\mid A\mid^2}$ and $t=2\pi-2\mbox{arcsin}\sqrt{1-\mid A\mid^2}$, in the interval $(0,2\pi)$. This contradicts to the statement of 
theorem~\ref{bosc} on the uniqueness of solution for the system (\ref{br}).

There is the following theorem 13 in \cite{BR} about the cut locus of the unit element $e\in SU(2)$ for lens spaces $L(p,q),$ where $SO(3)$ is diffeomorphic to $L(2,1)$.

\begin{theorem}
\label{br1} 
The cut locus for sub-Riemannian problem on $L(p,q)$ is a stratification
$$K_{[e]}=K^{sym}_{[e]}\cup K^{loc}_{[e]}$$
with
\begin{equation}
\label{ks}
K^{sym}_{[e]}=\left\{(A,B)\in SU(2): \R(A)\geq 0, \frac{\I(A)^2}{\sin^2(\pi/p)}+B^2=1\right\},
\end{equation}
\begin{equation}
\label{kl}
K^{loc}_{[e]}=\{(A,0)\in SU(2): A^{p}\neq 1\}.
\end{equation} 
\end{theorem}

Evidently, in the case $p=2,$ the equation in (\ref{ks}) is equivalent to the identity $A_1=\R(A)=0,$ and inequality in (\ref{kl}) is equivalent to inequality $A_2=\I(A)\neq 0.$ Now, using the formula (\ref{matr}), we can restate this theorem as follows. 

\begin{proposition}
\label{prel}
The cut locus for sub-Riemannian metric on $(SO(3),d)$ is a stratification
$$K_{[E]}=K^{sym}_{[E]}\cup K^{loc}_{[E]}$$
with
\begin{equation}
\label{ks1}
K^{sym}_{[E]}=
\left\{\left(\begin{array}{ccc}
A^2_2-B^2_1-B^2_2 & 2A_2B_1 & 2A_2B_2 \\
2A_2B_1 & -A^2_2+B^2_1-B^2_2 & 2B_1B_2 \\
2A_2B_2 & 2B_2B_1 & -A^2_2-B^2_1+B^2_2
\end{array}\right)\right\},
\end{equation}
\begin{equation}
\label{kl1}
K^{loc}_{[E]}=
\left\{\left(\begin{array}{ccc}
1 & 0 & 0 \\
0 & \R(A^2) & -\I(A^2) \\
0 & \I(A^2) & \R(A^2)
\end{array}\right)\right\},
\end{equation}
where $B=0,$ $\I(A)\neq 0$ in (\ref{kl1}). 
\end{proposition}

It is not difficult to prove that proposition \ref{prel} is equivalent
to the following theorem. 

\begin{theorem}
\label{inter}
The cut locus for sub-Riemannian metric on $(SO(3),d)$ is a stratification
$$K_{[E]}=K^{sym}_{[E]}\cup K^{loc}_{[E]}$$
with
\begin{equation}
\label{k1}
K^{sym}_{[E]}=
\left\{M\in SO(3)-E: M^2=E\right\},
\end{equation}
\begin{equation}
\label{l1}
K^{loc}_{[E]}=
\left\{\left(\begin{array}{cc}
1 & 0\quad 0 \\
(0\quad 0)^T & SO(2)\end{array}\right)\right\}-E.
\end{equation}
\end{theorem}

This is really very beautiful theorem. We note that the set (\ref{l1}) is the 
conjugate locus for $(SO(3),d).$

\end{document}